\documentclass[11pt]{article}

\usepackage{amsfonts}
\usepackage{amsthm}
\usepackage{amsmath}
\usepackage{amssymb}
\usepackage{mathtools} 
\usepackage{authblk} 
\usepackage{xcolor}
\usepackage{soul} 
\usepackage{mathrsfs}
\usepackage{graphicx}
\usepackage{caption}
\usepackage{subcaption}
\usepackage{float}
\usepackage{soul} 
\usepackage{mathrsfs}
\usepackage{cite}

\newtheorem{theorem}{Theorem}[section]
\newtheorem{lemma}{Lemma}[section]

\theoremstyle{definition}

\theoremstyle{remark}

\definecolor{cucol}{rgb}{0,0,0.8}
\definecolor{afcol}{rgb}{1,0,0}

\numberwithin{equation}{section}

\begin{document}
	\title{\bf Design of Optimal Controls in Acausal LQG Problems}
	\date{}
	\author{ Arzu Ahmadova\thanks{Faculty of Mathematics, University of Duisburg-Essen
			45127, Essen, Germany, E-mail: arzu.ahmadova@uni-due.de}\thanks{Corresponding author} and Agamirza E. Bashirov\thanks{Department of Mathematics, Eastern Mediterranean University,  Gazimagusa, Mersin 10, Turkey, Email: agamirza.bashirov@emu.edu.tr}}
	
	\maketitle
	\begin{abstract}
		In control theory, a system which has output depending only on the present and past values of the input is said to be causal (or nonanticipative). Respectively, a system is acausal (or non-causal) if its output depends on future inputs as well. Overall majority of literature in stochastic control theory discusses causal systems. Only a few sources indirectly concern acausal systems. In this paper, we systemize these results under main idea of acausality and present a background for designing optimal controls in acausal LQG problems.   
	\end{abstract}
	\textbf{Keywords:} Stochastic system, LQG problem, acausal system, stochastic control.\\
	\textbf{AMS Subject Classification:} 93E20, 93E11, 93E03

	
	\section{Introduction}
	\label{S.1}
	
	Overall majority of literature in stochastic control theory discusses causal stochastic control systems. Indeed, a signal-observation system in the general form
	\begin{equation}
		\label{1.1}
		\begin{cases}
			dx_t=f(t,x_t,u_t)\,dt+g(t,x_t,u_t)\,dw_t,\\
			dz_t=p(t,x_t,u_t)\,dt+q(t,x_t,u_t)\,dv_t,
		\end{cases}
	\end{equation}
	is causal if simply $w$ and $v$ are independent or correlated Wiener processes, which is a common condition under which stochastic control problems are considered. Control problems of this type are widely investigated and very important theoretical and applied results are achieved (see the resent books by Yong and Zhou \cite{YZ}, Fleming and Soner \cite{FS}, L\"{u} and Zhang \cite{LZ}, and references therein). However, the system in \eqref{1.1} becomes acausal if $w$ is a time delay of $v$, that is, $w_t=v_{t-\varepsilon}$ assuming that $g(t,x,u)\not= 0$, $q(t,x,u)\not= 0$, and $\varepsilon >0$. 
	
	To motivate the actuality of control systems with time-delay in noises, consider the scenario from \cite{B1} of a spacecraft at a fixed instant $t$. Assume that the radio signals reach it from the Earth and come back for a time $\varepsilon >0$. Generally, radio signals propagate safely in vacuum. However, they account noises at the beginning and at the end of the travel when they pass through the higher layers of the Earth’s atmosphere. Therefore, a ground radar detects at $t$ the signal 
	\[
	z_t=x_{t-\varepsilon /2}+w'_t
	\]
	about the state of the spacecraft at the instant $t-\varepsilon /2$ corrupted by the noise $w'_t$ which is modeled as a white noise. Now, assume that the control $u$ changes the position $x$ of the spacecraft in accordance to the linear equation $x'=Ax+Bu$ in which the noise effects and the distance to the spacecraft are neglected. However, if they are taken into account, the position of the spacecraft at $t-\varepsilon /2$ is changed by the control action that is sent by the ground radar at $t-\varepsilon $. This control propagating through the atmosphere accounts the noise $w'_{t-\varepsilon }$. Therefore, the equation for the position of the spacecraft must be updated as
	\[
	x'_{t-\varepsilon /2}=Ax_{t-\varepsilon /2}+B(u_{t-\varepsilon }+w'_{t-\varepsilon }).
	\]
	The substitution $\tilde{x}_t=x_{t-\varepsilon /2}$ and $\tilde{u}_t=u_{t-\varepsilon }$ leads to the signal-observation system
	\[
	\begin{cases}
		\tilde{x}'_t=A\tilde{x}_t+B\tilde{u}_t+Bw'_{t-\varepsilon },\\
		z_t=\tilde{x}_t+w'_t,
	\end{cases}
	\]
	in which the noise of the signal system is a delay of the observation noise.
	
	Since the altitude of the Earth orbiting satellites is negligible in comparison to the speed of radio signals, the time-delays can be disregarded in satellite communication. However, NASA's Mars Exploration Program (MEP) and other programs like Voyager-2 require communication with spacecraft significantly distanced from the Earth. Therefore, it is seen that the systems with time-delay of noises more adequately fit to the guidance, control, and dynamics problems of significantly distanced spacecraft. This justifies the importance of control problems for acausal systems obtained by pointwise delay of noises. 
	
	A signal-observation system becomes acausal also in the case if the state noise is a distributed delay of the observation noise. Indeed, a system of the form
	\[
	\begin{cases}
		x'_t=Ax_t+Bu_t+\varphi _t,\\
		dz_t=Cx_t\,dt+dv_t,
	\end{cases}
	\]
	where $\varphi $ is a distributed delay of $w$, that is 
	\[
	\varphi _t=\int _{\max (0,t-\varepsilon )}^t\Phi _{t,s-t}\,dw_s,
	\]
	is acausal for the two reasons. At first, a possible correlation of the Wiener processes $w$ and $v$ makes $v_t$ and $\varphi _s$ for $t\le s<t+\varepsilon $ to be correlated. Secondly, even if $w$ and $v$ are independent, $z_t$ contains information about future of $x$ because $\varphi _t$ and $\varphi _{s}$ for $|t-s|<\varepsilon $ are correlated while they are independent if $|t-s|\ge \varepsilon $, that is
	\begin{equation}
		\label{1.2}
		\mathrm{cov}\,(\varphi _t,\varphi _s)=
		\begin{cases}\lambda _{t,s}\not= 0 & \text{if}\ |t-s|<\varepsilon , \\ 0 & \text{if}\ |t-s|\ge \varepsilon .\end{cases}
	\end{equation} 
	
	In Fleming and Rishel \cite{FR}, p.\,126, a random process $\varphi $ with the property in \eqref{1.2} is called a wide band noise. In existing literature a wide band noise is also called as a bandwidth noise \cite{KR}.  Therefore, we will use for them the abbreviation BN while WN will abbreviate white noise. Although the WN model is more popular in the existing literature, BNs describe real noises more adequately. A few such scenarios are as follows:
	
	{\it Spacecraft communication}. Radio signals traveling from ground radars to spacecraft and coming back go through the Earth's atmosphere and account a communication noise. In the existing literature, this noise is modeled as a WN \cite{CJ}. The main sources creating communication noise are water vapor and radioactive particles in the upper layers of the atmosphere. It is natural to think that in a small time interval the density of water vapor and intensity of radiation change continuously. The weather forecasting which is based on the continuity of the density of water vapor is an example of support for this. This creates a correlation of values of a communication noise at instants within small time intervals. This makes it a BN with small $\varepsilon >0$.
	
	{\it Finance}. In finance, the noises are the results of many social and natural events that unexpectedly change the environment surrounding the financial sphere. Here, the effects of the social and natural events are quite more durable than in spacecraft communication. Therefore, the values of a financial noise at instants falling into visible intervals should be correlated. This makes it a BN with relatively large $\varepsilon >0$.
	
	{\it Quantization}. Quantization is the process of converting an analog signal to digital. It causes a partial loss of information, creating a quantization noise. Usually, a quantization noise is modeled in the form of a WN. In the quantization of a continuous signal, its positive value is likely to stay positive for a possibly short period of time. This creates a correlation between the values of quantization noise within this time interval and makes it a BN.
	
	{\it Epidemiology}. Modeling epidemic processes is very important for understanding how infectious diseases progress and helps prevent public interventions. In fact, the process opposite to quantization can be observed here. The digital numbers of susceptible and infectious individuals, which vary individually over time, are approximated by continuously changing values \cite{AA}. The difference between them produces noise very similar to quantization noise. Therefore, its values should be correlated within a small time intervals making it a BN.
	
	{\it Sensors}. A wide variety of sensors are used in applications to measure properties of materials. Sensor measurements take into account three types of noise: a noise from the sensor input (signal noise), a sensor noise normally modeled as WN, and a noise associated with the conversion of discrete measurements into continuous. Again, a process opposite to quantization noise is observed here making the third noise BN.
	
	These scenarios demonstrate that in a majority of cases the real noises behave as BN. Somehow, they can be treated as WN if $\varepsilon $ is too small. If, however, a real noise behaves essentially different from WNs, then the respective WN driven system may produce conclusions inconsistent with reality. This motivates study of control and filtering problems for acausal BN driven systems.
	
	There are two approaches to BNs developed independently from each other. The first approach  is directed to study nearly optimal controls for BN driven control systems \cite{KR, BP, K1, K3, LRT}. However, the second approach allows to synthesize optimal controls in the exact form. It is based on an integral representation of BNs \cite{B2, B3}. This idea was developed in \cite{B4, B5} and resulted with distributed delay structure of BNs in \cite{B6}. A discrete analog of BNs are known as a finite step autocorrelated noise and filtering problems for them have been studied in \cite{SZY, JZZ, LZW, STL, HWG, FWZ, ZZWZ, LSZ}.
	
	As far as we know, the first acausal control problem was considered in \cite{B7} in the frame of LQG problems. This was a purely control result without its ingredient filtering problem. LQG problems are famous in control theory because they can be solved completely if the supporting linear filtering result is provided. In the recent years, several filtering results were accumulated for acausal linear systems \cite{B8, B9, B10, B11, B12}.  Making a combination of control and filtering results and designing optimal feedback controls for different acausal LQG problems becomes possible. This paper reviews acausality in the frame of LQG problems, discusses control and filtering ingredients in acausal LQG problems, provides either main idea or presents a simplified complete proof of the results, and most importantly presents a complete set of equations for designing optimal feedback controls.  
	
	
	\section{General acausality in LQG problems}
	\label{S.2} 
	
	We consider the LQG problem of minimizing the quadratic functional 
	\begin{equation}
		\label{2.1}
		J(u)=\mathbf{E}\bigg( \langle x_T,Hx_T\rangle +\int _0^T(\langle x_t,Fx_t\rangle +\langle u_t,Gu_t\rangle )\,dt\bigg)
	\end{equation}
	subject to the partially observed linear system
	\begin{equation}
		\label{2.2}
		\begin{cases}
			dx_t=\big( Ax_t+Bu_t+\varphi ^1_t\big) \,dt+dw_t,\ x_0=\xi ,\ 0<t\le T,\\
			dz_t=\big( Cx_t+\varphi ^2_t\big) \,dt+dv_t,\ z_0=0,\, 0<t\le T,
		\end{cases}
	\end{equation}
	where $\langle \cdot , \cdot \rangle $ is a scalar product and $\mathbf{E}(\cdot )$ is an expectation. For simplicity, we assume that $A$, $B$, $C$, $F$, $H$, and $G$ are time-independent and nonrandom matrices of respective dimensions. We assume that $F$, $H$, and  $G$ are symmetric matrices so that $F$ and $H$ are positive semi-definite and $G$ is positive-definite (shortly, $F\ge 0$, $H\ge 0$, and $G>0$). If $w$ and $v$ are a pair of correlated or independent Wiener processes, $\varphi ^1_t=0$ and $\varphi ^2_t=0$, then the system in \eqref{2.2} is causal and the optimal control $u^*$ minimizing \eqref{2.1} has a linear feedback form
	\begin{equation}
		\label{2.3}
		u^*_t=-G^{-1}B^\top Q_t\mathbf{E}(x^*_t|z^*_s,0\le s\le t),
	\end{equation}
	where $x^*$ and $z^*$ are the state and observation processes, respectively, corresponding to $u^*$, $B^\top $ is the transpose of $B$, and $\mathbf{E}(\cdot |\cdot )$ is a conditional expectation. This result is called a separation principle which we will call as a {\it classical separation principle} because \eqref{2.3} changes in the case of acausal setting. Below within this section, we assume that the noise inputs of the system $\varphi ^1$, $\varphi ^2$, $w$ and $v$ are  independent or dependent in any fashion. Just it will be assumed that $\varphi ^1$ and $\varphi ^2$ are square integrable random processes, $w$ and $v$ are square integrable martingales, and $(\varphi ^1, \varphi ^2, w, v)$ and $\xi $ are independent to be sure of the existence of respective integrals. 
	
	
	\subsection{Admissible controls}
	\label{S.2.1} 
	
	Let $L_2(0,T;X)$ be a space of square integrable in the Lebesque sense functions on $[0,T]$ and let $L_2(\Omega ,X)$ be a space of square integrable random variables provided that $(\Omega ,\mathcal{F}, \mathbf{P})$ is an underlying probability space. In the case of sub-$\sigma $-field $\mathcal{F}'$ of $\mathcal{F}$, the space of $\mathcal{F}'$-measurable square integrable random variables will be denoted by $L_2(\Omega ,\mathcal{F}',\mathbf{P},X)$. 
	
	Assuming that the dimensions of the state, control, and observation spaces are $n$, $m$, and $k$, consider $u\in L_2(0,T;L_2(\Omega ,\mathbb{R}^m))$. Then the system in \eqref{2.2} defines the respective process $z\in L_2(0,T;L_2(\Omega ,\mathbb{R}^k))$ which will be denoted by $z^u$. Denote by $\mathcal{F}^u_t=\sigma (z^u_s,0\le s\le t)$ the minimal $\sigma $-field generated by $z^u_s$, $0\le s\le t$. Let 
	\[
	U^u_t=L_2(\Omega ,\mathcal{F}^u_t,\mathbf{P},\mathbb{R}^m),\ 0<t\le T.
	\]    
	In particular, $\mathcal{F}^u_t=\mathcal{F}^0_t$ and $U^u_t=U^0_t$ if $u\equiv 0$. Define the sets
	\[
	U=\{ u\in L_2(0,T;L_2(\Omega ,\mathbb{R}^m)):u_t\in U^u_t,\ 0<t\le T\} 
	\]
	and
	\[
	U^0=\{ u\in L_2(0,T;L_2(\Omega ,\mathbb{R}^m)):u_t\in U^0_t,\ 0<t\le T\} .
	\]
	Here, $U$ consists of controls $u$ which are adapted with respect to the filtration generated by the observations. Normally, it should be a set of admissible controls. However, $U$ is not a subspace of $L_2(0,T;L_2(\Omega ,\mathbb{R}^m))$ which makes it inconvenient. On the contrary, $U^0$ is a subspace of $L_2(0,T;L_2(\Omega ,\mathbb{R}^m))$ consisting of $\mathcal{F}^0=\{ \mathcal{F}^0_t: 0\le t\le T\} $-adapted random processes. However, the selection of $U^0$ as a set of admissible controls means disregarding the partial observations. Therefore, the set of admissible controls will be defined as
	\[
	U_{\rm ad}=U\cap U^0.
	\]
	In \cite{BV, CP}, the following properties related to $U_{\rm ad}$ are proved.
	\begin{itemize}
		\item[\rm (a)]
		$U_{\rm ad}$ contains all controls which can be represented in the linear feedback form 
		\[
		u_t=a_t+\int _0^tK_{t,s}\,dz^u_s
		\]
		where $a_t$ is deterministic and $K_{t,\cdot }\in L_\infty (0,t;\mathbb{R}^{m\times k})$ for all $0\le t\le T$. 
	\end{itemize}
	This ensures that $U_{\rm ad}$ is not empty containing the important class of controls in the linear feedback form. Additionally, 
	\begin{itemize}
		\item[\rm (b)]
		If $u\in U$, then $U^0_t\subseteq U^u_t$ for all $0<t\le T$. 
		\item[\rm (c)]
		If $u\in U^0$, then $U^u_t\subseteq U^0_t$ for all $0<t\le T$.  
	\end{itemize}
	Therefore,
	\begin{itemize}
		\item[\rm (d)]
		If $u\in U_{\rm ad}$, then $U^0_t=U^u_t$ for all $0<t\le T$.
	\end{itemize}
	However,
	\begin{itemize}
		\item[\rm (e)]
		$u\in U$ is not sufficient for $U^0_t=U^u_t$ for all $0<t\le T$.  
	\end{itemize}
	This justifies the selection of $U\cap U^0$ as a set of admissible controls since the equality $U^0_t=U^u_t$ plays a role in our derivation.
	
	For the further properties of $U_{\rm ad}$, define
	\[
	U_\varepsilon =\{ u\in L_2(0,T;L_2(\Omega ,\mathbb{R}^m)):u_t\in U^u_{t-\varepsilon },\ 0<t\le T\} , 
	\]
	where $\varepsilon >0$ and $U^u_{t-\varepsilon }=\mathbb{R}^m$ if $0<t\le \varepsilon $. In \cite{BV}, it is proved that
	\begin{itemize}
		\item[\rm (f)]
		If $\varepsilon >0$ and $u\in U_\varepsilon $, then $U^0_t=U^u_t$ for all $0<t\le T$.
	\end{itemize}
	Respectively,
	\begin{itemize}
		\item[\rm (g)] $U_\varepsilon \subseteq U_{\rm ad}$ for all $\varepsilon >0$. 
	\end{itemize}
	This ensures that $U_{\rm ad}$ is quite large set including not only controls in the linear feedback form. Additionally,
	\begin{itemize}
		\item[\rm (h)] $U_{\rm ad}$ is dense in $U^0$.
	\end{itemize}
	This implies
	\begin{itemize}
		\item[\rm (i)] $\inf _{U_{\rm ad}}J(u)=\inf _{U^0}J(u)$, where $J$ is defined by \eqref{2.1}. 
	\end{itemize}
	Below, we use the short notation $\mathbf{E}^u_t=\mathbf{E}(\,\cdot \,|\mathcal{F}^u_t)$. In particular, $\mathbf{E}^0_t=\mathbf{E}(\,\cdot \,|\mathcal{F}^0_t)$.
	
	
	\subsection{Extended separation principle}
	\label{S.2.2}  
	
	\begin{lemma}
		\label{L.2.1}
		Assume that $A,H,F\in \mathbb{R}^{n\times n}$, $B\in \mathbb{R}^{n\times m}$, $C\in \mathbb{R}^{k\times n}$, $G\in \mathbb{R}^{m\times m}$, $H\ge 0$, $F\ge 0$, $G>0$, $w$ and $v$ are $n$- and $k$-dimensional martingales, $\varphi ^1\in L_2(0,T;L_2(\Omega ,\mathbb{R}^n))$, $\varphi ^2\in L_2(0,T;L_2(\Omega ,\mathbb{R}^k))$, $(\varphi ^1, \varphi ^2, w, v)$ and $\xi $ are independent. Additionally, assume that the functional $J$ from \eqref{2.1} subject to \eqref{2.2} takes its minimum on $U^0$ at $u^*\in U^0$. Then
		\begin{equation}
			\label{2.4}
			u^*_t=G^{-1}B^\top \mathbf{E}^0_ty_t,\ 0\le t\le T,
		\end{equation}
		where 
		\begin{equation}
			\label{2.5}
			y_t=-e^{A^\top (T-t)}Hx^*_T-\int _t^Te^{A^\top (s-t)}Fx^*_s\,ds
		\end{equation}
		and $x^*$ is the state process corresponding to the control $u^*$.
	\end{lemma}
	\begin{proof}
		Take $u\in U^0$. We have $u^*+\lambda u\in U^0$ for all $\lambda \in \mathbb{R}$ since $U^0$ is closed under the linear combinations. Denote
		\[
		h_t=\int _0^te^{A(t-s)}Bu_s\,ds.
		\]
		Then
		\begin{align*}
			0 &\le J(u^*+\lambda u)-J(u^*)\\
			&=2\lambda \mathbf{E}\bigg( \langle x^*_T,Hh_T\rangle +\int _0^T(\langle x^*_t,Fh_t\rangle +\langle u^*_t,Gu_t\rangle )\,dt\bigg) \\
			&\ \ \ +\lambda ^2\mathbf{E}\bigg( \langle x^*_T,Hx^*_T\rangle +\int _0^T(\langle x^*_t,Fx^*_t\rangle +\langle u^*_t,Gu^*_t\rangle )\,dt\bigg) .   
		\end{align*} 
		Dividing both sides consequently by $2\lambda $ and by $-2\lambda $ and then taking the limit when $\lambda \to 0$, we obtain
		\[
		0=\mathbf{E}\bigg( \langle x^*_T,Hh_T\rangle +\int _0^T(\langle x^*_t,Fh_t\rangle +\langle u^*_t,Gu_t\rangle )\,dt\bigg) .
		\] 
		This implies
		\begin{align*}
			0&=\mathbf{E}\int _0^T\big( \big\langle x^*_T,He^{A(T-t)}Bu_t\big\rangle +\langle u^*_t,Gu_t\rangle \big) \,dt\\
			&\ \ \ +\mathbf{E}\int _0^T\int _0^t\big( \big\langle x^*_t,Fe^{A(t-s)}Bu_s\big\rangle \big) \,dsdt.
		\end{align*}
		Interchanging the order of integration in the repeated integrals and using the arbitrariness of $u$ in $U^0$, we arrive to
		\[
		0=Gu^*_t+\mathbf{E}^0_t\bigg( B^\top e^{A^\top (T-t)}Hx^*_T+\int _t^TB^\top e^{A^\top (s-t)}Fx^*_s\,ds\bigg) .
		\]
		This proves \eqref{2.4} and \eqref{2.5}.
	\end{proof}
	\begin{lemma}
		\label{L.2.2}
		Assume that $A,H,F\in \mathbb{R}^{n\times n}$, $B\in \mathbb{R}^{n\times m}$, $C\in \mathbb{R}^{k\times n}$, $G\in \mathbb{R}^{m\times m}$, $H\ge 0$, $F\ge 0$, $G>0$, $w$ and $v$ are $n$- and $k$-dimensional martingales, $\varphi ^1\in L_2(0,T;L_2(\Omega ,\mathbb{R}^n))$, $\varphi ^2\in L_2(0,T;L_2(\Omega ,\mathbb{R}^k))$, $(\varphi ^1, \varphi ^2, w, v)$ and $\xi $ are independent. Additionally, assume that the functional $J$ from \eqref{2.1} subject to \eqref{2.2} takes its minimum on $U_{\rm ad}$ at $u^*\in U_{\rm ad}$. Let $x^*$ be the state process corresponding to the control $u^*$ and define $y$ by \eqref{2.5}. Then the following equality holds
		\begin{equation}
			\label{2.6}
			\mathbf{E}^0_\tau (y_t+K_tx^*_t+\alpha _t)=0,\ 0<\tau \le t\le T, 
		\end{equation}
		where $K$ is a solution of the backward Riccati equation
		\begin{equation}
			\label{2.7}
			\begin{cases}
				-K'_t=K_tA+A^\top K_t+F-K_tBG^{-1}B^\top K_t,\\
				K_T=H,\  0\le t<T.
			\end{cases}
		\end{equation}
		and $\alpha $ is a solution of the backward linear stochastic differential equation
		\begin{equation}
			\label{2.8}
			\begin{cases}
				-d\alpha _t=((A^\top -K_tBG^{-1}B^\top )\alpha _t+K_t\varphi _t)\,dt+K_t\,dw_t,\\
				\alpha_T=0,\  0\le t<T.
			\end{cases}
		\end{equation}
	\end{lemma}
	\begin{proof}
		At first, note a solution of \eqref{2.8} is a random process $\alpha $ satisfying the integral equation 
		\begin{align}
			\label{2.9}
			\alpha _t&=\int _t^Te^{A^\top (s-t)}K_s\big( \varphi _s-BG^{-1}B^\top \alpha _s\big) ds+\int _t^Te^{A^\top (s-t)}K_s\, dw_s,
		\end{align}
		where $e^{At}$, $t\ge 0$, is a transition matrix of $A$. If\, $\mathcal{U}_{t,s}$, $0\le s\le t$, is a transition matrix of $A-BG^{-1}B^\top K_t$, then this equation has a unique solution in the exact form
		\[
		\alpha _t=\int _t^T\mathcal{U}^\top _{t,s}K_s\varphi _s\,ds+\int _t^T\mathcal{U}^\top _{t,s}K_s\,dw_s,\ 0\le t\le T.
		\]
		
		Next, since $U_{\rm ad}$ is a dense subset of $U^0$, the functional $J$ takes its minimum on $U^0$ at $u^*$ as well. Therefore, by Lemma \ref{L.2.1}, the equalities in \eqref{2.4} and \eqref{2.5} hold. The solution of the state equation in \eqref{2.2} corresponding to $u^*$ can be written as
		\begin{equation}
			\label{2.10}
			x^*_t=e^{A(t-\tau )}x^*_\tau +\int _\tau ^te^{A(t-s)}(Bu^*_s+\varphi _s)\,ds +\int _\tau ^te^{A(t-s)}\,dw_s
		\end{equation}
		for $0\le \tau \le t\le T$. Substituting \eqref{2.4} in \eqref{2.10}, we obtain 
		\begin{align*}
			x^*_t
			&=e^{A(t-\tau )}x^*_\tau +\int _\tau ^te^{A(t-s)}\big( BG^{-1}\mathbf{E}^0_sB^\top y_s+\varphi _s\big) \,ds+\int _\tau ^te^{A(t-s)}\,dw_s.
		\end{align*}
		Since $\mathbf{E}^0_s\mathbf{E}^0_\tau =\mathbf{E}^0_\tau $ for $0\le \tau \le s\le T$, this implies
		\begin{align}
			\label{2.11}
			\mathbf{E}^0_\tau x^*_t&=\mathbf{E}^0_\tau \bigg( e^{A(t-\tau )}x^*_\tau +\int _\tau ^te^{A(t-s)}\big( BG^{-1}B^\top y_s+\varphi _s\big) \,ds\nonumber \\
			&\ \ \ +\int _\tau ^te^{A(t-s)}\,dw_s\bigg) .
		\end{align}
		
		\eqref{2.11}, \eqref{2.5}, and \eqref{2.9} will be used to prove \eqref{2.6}. For this, we derive an expression for $\mathbf{E}^0_{\tau }K_tx^*_t$ in the following way. The backward Riccati equation in \eqref{2.7} can be written in the integral form
		\[
		K_t=e^{A^\top (T-t)}He^{A(T-t)}+\int _t^Te^{A^\top (s-t)}\tilde{F}_se^{A(s-t)}\,ds,
		\]
		where for brevity we denote
		\[
		\tilde{F}_t=F-K_tBG^{-1}B^\top K_t.
		\]
		Also, for brevity, let  
		\[
		h_t=BG^{-1}B^\top y_t+\varphi_t.
		\]
		Then
		\begin{align*}
			\mathbf{E}^0_{\tau }K_tx^*_t
			&=\mathbf{E}^0_{\tau }\bigg( e^{A^\top (T-t)}He^{A(T-t)}x^*_t+\int ^T_te^{A^\top (s-t)}\tilde{F}_se^{A(s-t)}x^*_t\, ds\bigg) \\
			&=\mathbf{E}^0_{\tau }\bigg( e^{A^\top (T-t)}H\bigg( x^*_T-\int _t^Te^{A(T-r)}\,dw_r-\int _t^Te^{A(T-r)}h_r\,dr\bigg) \\
			&\ \ \ +\int ^T_te^{A^\top (s-t)}\tilde{F}_s\bigg( x^*_s-\int _t^se^{A(s-r)}\, dw_r-\int _t^se^{A(s-r)}h_r\,dr\bigg) ds\bigg) .
		\end{align*}
		Grouping the terms in the right side, we obtain 
		\begin{align*}
			\mathbf{E}^0_{\tau }Q_tx^*_t
			&=\mathbf{E}^0_{\tau }\bigg( e^{A^\top (T-t)}Hx^*_T+ \int ^T_te^{A^\top (s-t)}\tilde{F}_sx^*_s\, ds \\
			&\ \ \ -\int ^T_t\!\!\!e^{A^\top (r-t)}\bigg( e^{A^\top (T-r)}He^{A(T-r)}\!+\!\!\int ^T_r\!\!e^{A^\top (s-r)}\tilde{F}_se^{A(s-r)}\,ds\bigg) dw_r \\
			&\ \ \ -\int ^T_t\!\!e^{A^\top (r-t)}\!\bigg( \!e^{A^\top (T-r)}He^{A(T-r)}\!+\!\!\int _r^T\!\!e^{A^\top (s-r)}\tilde{F}_se^{A(s-r)}\,ds\!\bigg) h_r\,dr\!\!\bigg) .
		\end{align*}
		This implies
		\begin{align*} 
			\mathbf{E}^0_{\tau }K_tx^*_t
			&=\mathbf{E}^0_{\tau }\bigg( e^{A^\top (T-t)}Hx^*_T+\int ^T_te^{A^\top (s-t)}\tilde{F}_sx^*_s\, ds\\
			&\ \ \ -\int ^T_te^{A^\top (s-t)}K_s\,dw_s-\int ^T_te^{A^\top (s-t)}K_sh_s\,ds\bigg) .
		\end{align*}
		Thus,
		\begin{align}
			\label{2.12}
			\mathbf{E}^0_{\tau }K_tx^*_t
			&=\mathbf{E}^0_{\tau }\bigg( e^{A^\top (T-t)}Hx^*_T- \int _t^Te^{A^\top (s-t)}K_s\, dw_s\nonumber \\
			&\ \ \ +\int ^T_te^{A^\top (s-t)}\left( F-K_sBG^{-1}B^\top K_s\right) x^*_s\,ds\nonumber \\
			&\ \ \ -\int _t^Te^{A^\top (s-t)}K_s\big( BG^{-1}B^\top y_s+\varphi _s\big) ds\bigg) .
		\end{align}
		Using \eqref{2.5}, \eqref{2.9}, and \eqref{2.12}, we arrive to
		\[
		\mathbf{E}^0_{\tau }(y_t+K_tx^*_t+\alpha _t)=
		-\int _t^Te^{A^\top (s-t)}K_sBG^{-1}B^\top \mathbf{E}^0_\tau (y_s+K_sx^*_s+\alpha _s)\, ds.
		\]
		If
		\[
		\lambda _{t,\tau }=\left\| \mathbf{E}^0_{\tau }(y_t+K_tx^*_t+ \alpha _t)\right\| ,
		\]
		then
		\[
		\lambda _{t,\tau }\le c\int _t^T\lambda _{s,\tau }\, ds,\
		c=\mathrm{const.}\ge 0,
		\]
		which, by the Gronwall's inequality, implies $\lambda _{t,\tau }=0$, $0\le \tau \le t\le T$. This completes the proof.
	\end{proof}
	\begin{theorem}[Extended Separation Principle]
		\label{T.2.1}
		Assume that $A,H,F\in \mathbb{R}^{n\times n}$, $B\in \mathbb{R}^{n\times m}$, $C\in \mathbb{R}^{k\times n}$, $G\in \mathbb{R}^{m\times m}$, $H\ge 0$, $F\ge 0$, $G>0$, $w$ and $v$ are $n$- and $k$-dimensional martingales, $\varphi ^1\in L_2(0,T;L_2(\Omega ,\mathbb{R}^n))$, $\varphi ^2\in L_2(0,T;L_2(\Omega ,\mathbb{R}^k))$, $(\varphi ^1, \varphi ^2, w, v)$ and $\xi $ are independent. Additionally, assume that the functional $J$ from \eqref{2.1} subject to \eqref{2.2} takes its minimum on $U_{\rm ad}$ at $u^*\in U_{\rm ad}$. Denote by $x^*$ the state processes corresponding to the optimal control $u^*$, and let $\mathbf{E}^*_t=\mathbf{E}(\cdot |\mathcal{F}^{u^*}_t)$.
		Then 
		\begin{equation}
			\label{2.13}
			u^*_t=-G^{-1}B^\top K_t\mathbf{E}^*_tx^*_t-G^{-1}B^\top \mathbf{E}^*_t\alpha _t,\ 0\le t\le T,
		\end{equation}
		where $K$ is a unique solution of \eqref{2.7} and $\alpha $ is defined by \eqref{2.8}.
	\end{theorem}
	\begin{proof}
		Since $u^*\in U_{\rm ad}$, we have $U^{u^*}_\tau =U^0_\tau $ for all $0\le \tau \le T$. Therefore, $\mathbf{E}^0_\tau $ in \eqref{2.6} can be replaced by $\mathbf{E}^*_\tau $.  Then for $\tau =t$ in \eqref{2.6}, we obtain
		\begin{equation}
			\label{2.14}
			\mathbf{E}^*_ty_t=-\mathbf{E}^*_t(K_tx^*_t+\alpha _t),\ 0<t\le T. 
		\end{equation}
		Substituting \eqref{2.14} in \eqref{2.4}, we arrive to \eqref{2.13}.
	\end{proof}
	
	
	\subsection{Concluding remarks}
	\label{S.2.3} 
	
	Theorem \ref{T.2.1} does not state the existence of the optimal control. However, \eqref{2.13} and property (a) of the set of admissible controls say that if the best estimates $\mathbf{E}^*_tx^*_t$ and $\mathbf{E}^*_t\alpha _t$ accept a linear feedback form then the optimal control exists. A few such cases will be reviewed in the consequent sections in which we will adopt the notation
	\[
	\hat{x}^*_t=\mathbf{E}^*_tx^*_t\ \text{and}\ \hat{\alpha }_t=\mathbf{E}^*_t\alpha _t
	\]
	from filtering theory. Then \eqref{2.13} can be written as
	\begin{equation}
		\label{2.15}
		u^*_t=-G^{-1}B^\top K_t\hat{x}^*_t-G^{-1}B^\top \hat{\alpha }_t,\ 0\le t\le T.
	\end{equation}
	
	Theorem \ref{T.2.1} demonstrates that the optimal control in the LQG problem under rather general conditions on the relationship of the random inputs accounts the new term 
	\begin{equation}
		\label{2.16}
		u^{1*}_t=-G^{-1}B^\top \hat{\alpha }_t
	\end{equation}
	in addition to the term 
	\begin{equation}
		\label{2.17}
		u^{0*}_t=-G^{-1}B^\top K_t\hat{x}^*_t.
	\end{equation}
	This is due to acausality of the partially observed system under consideration. 
	
	What is the role of the new term in the problem? For this, substitute $u_t=u^{0*}_t+u^{1*}_t$ in \eqref{2.2} and combine the state noise in \eqref{2.2} with the $Bu^{1*}_t$. This  defines a new noise process $\phi _t$ by
	\[
	d\phi _t=\varphi ^1_t\,dt+dw_t+Bu^{1*}_t\,dt
	\]
	or by 
	\[
	d\phi _t=\big( \varphi ^1_t-BG^{-1}B^\top \hat{\alpha }_t\big) \,dt+dw_t,
	\]
	the role of which could be considered as a dual to the role of the innovation process in the filtering problems. If the process $\phi $ is given beforehand, then the system becomes causal, and the optimal control falls into the classic form given by \eqref{2.17}. This is a good hint to handle nonlinear acausal systems. Thus, the obligation on the optimal controls in the acausal stochastic control problems is seen to be twofold. 
	\begin{itemize}
		\item At first, they should have a component-term like $u^{1*}$ in \eqref{2.16} which transforms acausal setting to causal. 
		\item
		Secondly, they should have another component-term like $u^{0*}$ in \eqref{2.17} which minimizes the cost functional in the transformed causal problem. 
	\end{itemize}
	
	In order to complete the design issue for the control in \eqref{2.15}, effective filtering results for $\hat{x}^*_t$ and $\hat{\alpha }_t$ must be provided.  
	This is done below in a few specific acausality cases.
	
	
	\section{Case 1: State is corrupted by BN}
	\label{S.3} 
	
	\subsection{Setting of problem}
	
	We consider a particular case of the system in \eqref{2.2} in the form
	\begin{equation}
		\label{3.1}
		\begin{cases}
			x'_t=Ax_t+Bu_t+\varphi _t,\ x_0=\xi ,\ 0<t\le T, \\
			dz_t=Cx_t\,dt+dv_t,\ z_0=0,\ 0<t\le T.
		\end{cases}
	\end{equation}
	In addition to the conditions of Theorem \ref{T.2.1}, we assume that $\varphi $ is an $n$-dimensional BN with the autocovariance function 
	\[
	\Lambda _{t,\theta }=\mathrm{cov}\,(\varphi _{t+\theta },\varphi _t),\ 0\le \theta \le \varepsilon ,\ t\ge 0.
	\]
	while $v$ is a $k$-dimensional standard Wiener process. The noise inputs of the system $\varphi $, $v$ and $\xi $ are assumed to be mutually independent, $\xi $ is Gaussian, and $\mathbf{E}\xi =0$. Note that the signal system in \eqref{3.1} is written in terms of derivative unlike \eqref{2.2} which it is written in terms of differentials. This is because a BN is an ordinary random process unlike WN which is a generalized random process.
	
	\subsection{Optimal control}
	
	A combination of Theorem \ref{T.2.1} and the filtering result from \cite{B11, B10} produces the existence of a unique optimal control $u^*$ minimizing the cost functional in \eqref{2.1} subject to \eqref{3.1}. It is defined by \eqref{2.15} where $\hat{x}^*$ is a unique solution of the equation 
	\begin{equation}
		\label{3.2}
		\begin{cases}
			d\hat{x}^*_t=(A\hat{x}^*_t+Bu^*_t+\psi _{t,0})\,dt+P_tC^\top (dz^*_t-C\hat{x}^*_t\,dt),\\
			\hat{x}^*_0=0,\ 0< t\le T,
		\end{cases}
	\end{equation}
	and $\hat{\alpha }$ has the explicit form
	\begin{equation}
		\label{3.3}
		\hat{\alpha }_t=\int _t^{\min (t+\varepsilon , T)}\mathcal{U}^\top _{s,t}K_s\psi _{t,t-s}\,ds,
	\end{equation}
	provided that $\mathcal{U}_{t,s}$ is the transition matrix of $A-BG^{-1}B^\top K_t$. The random process $\psi $ included to \eqref{3.2} and \eqref{3.3} is a unique mild solution of the equation
	\begin{equation}
		\label{3.4} 
		\begin{cases} 
			\big(\frac{\partial }{\partial t}+\frac{\partial }{\partial \theta }\big) \psi _{t,\theta }\,dt=Q_{t,\theta }C^\top (dz^*_t-C\hat{x}^*_t\,dt),\\
			\psi _{0,\theta }=\psi _{t,-\varepsilon }=0,\ -\varepsilon \le \theta \le 0,\ 0<t\le T,
		\end{cases} 
	\end{equation}
	Additionally, $K$ is a unique solution of the backward Riccati equation in \eqref{2.7}, $P$ is a unique solution of the forward Riccati equation
	\begin{equation}
		\label{3.5}
		\begin{cases}
			P'_t=AP_t+P_tA^\top +Q_{t,0}+Q^\top _{t,0}-P_tC^\top CP_t,\\
			P_0=\mathrm{cov}\,\xi ,\ 0<t\le T,
		\end{cases}
	\end{equation}
	$Q$ and $R$ are unique mild solutions of
	\begin{equation}
		\label{3.6}
		\begin{cases} 
			\big( \frac{\partial }{\partial t}+\frac{\partial }{\partial \theta} \big)Q_{t,\theta }=Q_{t,\theta }A^\top +\Lambda _{t,-\theta }-R_{t,\theta ,0}-Q_{t,\theta }C^\top CP_t,\\
			Q_{0,\theta }=Q_{t,-\varepsilon }=0,\ -\varepsilon \le \theta \le 0,\ 0<t\le T,
		\end{cases} 
	\end{equation}
	and
	\begin{equation}
		\label{3.7}
		\begin{cases} 
			\big( \frac{\partial }{\partial t}+\frac{\partial }{\partial \theta}+\frac{\partial }{\partial \tau }\big) R_{t,\theta ,\tau }=Q_{t,\theta }C^\top CQ^\top _{t,\tau },\\
			R_{0,\theta ,\tau }\!=\!R_{t,-\varepsilon ,\tau }\!=\!R_{t,\theta ,-\varepsilon }\!=\!0,\ -\varepsilon \!\le \!\theta \!\le \!0,\ -\varepsilon \!\le \!\tau \!\le \!0,\ 0\!<t\!\le T.
		\end{cases} 
	\end{equation}
	
	Here, \eqref{3.4}, \eqref{3.6}, and \eqref{3.7} are partial differential equations and, therefore, they are infinite dimensional. Their solutions are understood in the mild sense. For \eqref{3.4}, the mild solution has the explicit form  
	\[
	\psi _{t,\theta }=\int _{\max (0,t-\theta -\varepsilon )}^tQ_{s,s-t+\theta }C^\top \,(dz^*_s-C\hat{x}^*_s\,ds). 
	\]
	However, for \eqref{3.6} and \eqref{3.7}, the mild solutions are the solutions of the integral equations, respectively, 
	\[
	Q_{t,\theta }=\int _{\max (0,t-\theta -\varepsilon )}^t(Q_{s,s-t+\theta }A^\top +\Lambda _{s,t-s-\theta }-R_{s,s-t+\theta ,0}-Q_{s,s-t+\theta }C^\top CP_s)\,ds 
	\]
	and
	\[
	R_{t,\theta ,\tau }=\int _{\max (0,t-\theta -\varepsilon ,t-\tau -\varepsilon )}^tQ_{s,s-t+\theta }C^\top CQ^\top _{s,s-t+\tau }\,ds. 
	\]
	
	\subsection{Sketch of the proof}
	\label{S.3.3}
	
	How are these equations obtained? The idea of the derivation is based on the integral representation
	\begin{equation}
		\label{3.8}
		\varphi _t=\int _{\max (0,t-\varepsilon )}^t\Phi _{t,s-t}\,dw_s
	\end{equation}
	where $w$ is an $l$-dimensional standard Wiener process, $\Phi $ is a function in $L_2(0,T;L_2(-\varepsilon ,0;\mathbb{R}^{n\times l}))$ obtained from equating the autocovariance function of \eqref{3.8} to $\Lambda _{t,\theta }$. This function is called a relaxing (damping) function. Generally, there are infinitely many relaxing functions $\Phi $ so that the autocovariance function of \eqref{3.8} equals to $\Lambda _{t,\theta }$. Take one of them and consider the family of BNs 
	\[
	\tilde{\varphi }_{t,\theta }=\int _{\max (0,t-\theta -\varepsilon )}^t\Phi _{t-\theta ,s-t+\theta }\,dw_s
	\] 
	for the family parameter $-\varepsilon \le \theta \le 0$. It satisfies $\Gamma \tilde{\varphi }_{t,\cdot }=\tilde{\varphi }_{t,0}=\varphi _t$ and 
	\begin{equation}
		\label{3.9}
		\begin{cases}
			\big(\frac{\partial }{\partial t}+\frac{\partial }{\partial \theta }\big) \tilde{\varphi }_{t,\theta }\,dt=\Phi _{t-\theta ,\theta }\,dw_t,\\
			\tilde{\varphi }_{0,\theta }=\tilde{\varphi }_{t,-\varepsilon }=0.
		\end{cases}
	\end{equation}
	Then, the state equation in \eqref{3.1} can be combined with \eqref{3.9} and the result becomes a linear stochastic equation 
	\begin{equation}
		\label{3.10}
		d\begin{bmatrix}x_t\\ \tilde{\varphi }_{t,\cdot}\end{bmatrix}=
		\begin{bmatrix}A & \Gamma \\ 0 & -d/d\theta \end{bmatrix}\begin{bmatrix}x_t\\ \tilde{\varphi }_{t,\cdot}\end{bmatrix}
		+\begin{bmatrix}B\\ 0\end{bmatrix}u_t+\begin{bmatrix}0\\ \Phi _{t-\,\cdot ,\cdot}\end{bmatrix}\,dw_t.
	\end{equation}
	for the new state process $[\, x_t\,\,\tilde{\varphi }_{t,\cdot }\,]^\top $. The cost functional in \eqref{2.1} and the observation sysem in \eqref{3.1} can also be written in terms of the new state process. Thus, the original acausal LQG problem is reduced to a new causal LQG problem for the WN driven state equation  while it is infinite dimensional. Applying the separation principle and Kalman filter in infinite dimensions \cite{CP}, the optimal control can be synthesized in the linear feedback form through forward and backward Riccati equations which are infinite dimensional. The rest of the proof is getting equations for the components of the optimal state and Riccati equations and making a specific substitution which transforms the dependence of equations on $\Phi $ to the dependence on $\Lambda $.
	
	\subsection{Discussion} 
	
	\eqref{2.15}, \eqref{3.2}--\eqref{3.7}, and \eqref{2.7} form a complete set of equations for the optimal control in the LQG problem for the cost functional from \eqref{2.1} and the BN driven state system from \eqref{3.1}. What do these equations mean? 
	\begin{itemize}
		\item 
		At first, note that $\bar{z}$, defined by 
		\[
		d\bar{z}_t=dz_t-C\hat{x}_t\,dt,\ \bar{z}_0=0,\ t>0,
		\] 
		is a Wiener process in accordance to Kalman filtering. It is called an innovation process.
		\item
		Comparing the equations in \eqref{3.4} and \eqref{3.9}, it is seen that $\psi _{t,0}$ is a BN generated by the Wiener process $\bar{z}$. So, the equation in \eqref{3.2} for the best estimate $\hat{x}^*_t$ is driven by the sum of WN $\bar{z}'_t$ and BN $\psi _{t,0}$, where the latter is generated by \eqref{3.4}. 
		\item
		The equation in \eqref{2.7} is a backward Riccati equation associated with control problem while the equation in \eqref{3.5} is a modification of the forward Riccati equation from Kalman filtering. 
		\item
		The equation in \eqref{3.7} for $R$ generates the autocovariance function $\Sigma _{t,\theta } =R_{t,-\theta ,0}$ of $\varphi _{t,0}$.
		\item 
		The relaxing function of $\psi _{t,0}$ has the exact form  $\Psi _{t,\theta }=Q_{t+\theta ,\theta }C^\top $
		which is generated by equation \eqref{3.6} for $Q$.
		\item
		The innovative noise $\phi $ making the system \eqref{3.1} causal is defined by
		\[
		\phi '_t=\varphi _t-BG^{-1}B^\top \hat{\alpha }_t
		\]
		or
		\[
		\phi '_t=\varphi _t-BG^{-1}B^\top \int _t^{\min (t+\varepsilon ,T)}\mathcal{U}^\top _{s,t}K_s\psi _{t,t-s}\,ds.
		\]
		\item 
		The equations in \eqref{2.15}, \eqref{3.2}--\eqref{3.7} and \eqref{2.7} of the optimal control are independent on the variations of the relaxing functions $\Phi $ of $\varphi $. They depend only on the autocovariance function $\Lambda $ of $\varphi $. This property is called an invariance. However, for the BN $\psi _{t,0}$ corrupting the equation of the optimal state these equations create the relaxing function $\Psi _{t,\theta }$, demonstrating that a relaxing function is an important parameter of BNs. It seems the independence on $\Phi $ is a consequence of the linearity of \eqref{3.1}. For nonlinear systems, it may not be valid.  
	\end{itemize}
	
	
	\section{Case 2: State and observations are corrupted by BN}
	\label{S.4}
	
	\subsection{Setting of problem} 
	
	Acausality takes place if the observations are corrupted by a BN as well. Consider the system in \eqref{2.2} in the form
	\begin{equation}
		\label{4.1}
		\begin{cases}
			x'_t=Ax_t+Bu_t+\varphi ^1_t,\ x_0=\xi ,\ 0<t\le T,\\
			dz_t=\big( Cx_t+\varphi ^2_t\big) \,dt+dv_t,\ z_0=0,\, 0<t\le T.
		\end{cases}
	\end{equation}
	In addition to the conditions of Theorem \ref{T.2.1}, we assume that $\varphi ^1$ and $\varphi ^2$ are $n$ and $k$-dimensional BNs with the autocovariance and cross covariance functions
	\[
	\Lambda ^{11}_{t,\theta }=\mathrm{cov}\,(\varphi ^1_{t+\theta },\varphi ^1_t),\  
	\Lambda ^{22}_{t,\alpha }=\mathrm{cov}\,(\varphi ^2_{t+\alpha },\varphi ^2_t),\ 
	\Lambda ^{12}_{t,\theta  }=\mathrm{cov}\,(\varphi ^1_{t+\theta },\varphi ^2_t).
	\]
	while $v$ is a $k$-dimensional standard Wiener process, $(\varphi ^1,\varphi ^2)$, $v$, and $\xi $ are mutually independent, $\xi $ is Gaussian, and $\mathbf{E}\xi =0$. 
	
	Note that in \eqref{4.1} the observation noise is modeled as the sum of BN and WN. The presence of WN is a restriction coming from Kalman filtering. At the first glance, such a combination of noises is not realistic. However, sensor measurements quite fit to the observation system in \eqref{4.1}, where the terms $Cx_t$, $v'_t$, and $\varphi ^2_t$ reflect a noise from the sensor input, a sensor noise related to deviation of sensor parameters from the standards, and a noise associated with the conversion of discrete measurements into continuous, respectively. Generally, a sensor noise is modeled as a WN while the discrete to continuous conversion creates a BN as it was discussed in the introductory section. 
	
	\subsection{Optimal control}
	
	A combination of Theorem \ref{T.2.1} and the filtering result from \cite{B8, B10} produces the existence of a unique optimal control $u^*$ minimizing the cost functional in \eqref{2.1} subject to t\eqref{4.1}. It is defined by \eqref{2.15} where $\hat{x}^*$ is a unique solution of the equation 
	\begin{equation}
		\label{4.2}
		\begin{cases}
			d\hat{x}^*_t=(A\hat{x}^*_t+\psi ^1_{t,0}+Bu^*_t)\,dt+(P_tC^\top +M^\top _{t,0})\\
			\ \ \ \ \ \ \ \ \times (dz^*_t-C\hat{x}^*_t\,dt -\psi ^2_{t,0}\,dt),\\
			\hat{x}^*_0=0,\ 0<t\le T,
		\end{cases}
	\end{equation} and $\hat{\alpha }$ has the explicit form
	\begin{equation}
		\label{4.3}
		\hat{\alpha }_t=\int _t^{\min (t+\varepsilon ,T)}\mathcal{U}^\top _{s,t}K_s\psi ^1_{t,t-s}\,ds,
	\end{equation}
	provided that $\mathcal{U}_{t,s}$ is the transition matrix of $A-BG^{-1}B^\top K_t$. The random processes $\psi ^1$ and $\psi ^2$ included to \eqref{4.2} and \eqref{4.3} are unique mild solutions of the equations
	\begin{equation}
		\label{4.4}
		\begin{cases}
			\big( \frac{\partial }{\partial t}+\frac{\partial }{\partial \theta }\big) \psi ^1_{t,\theta }\,dt=\big( Q_{t,\theta }C^\top +\Lambda ^{12}_{t,-\theta }-S_{t,\theta ,0}\big) \\
			\ \ \ \ \ \ \ \ \ \ \ \ \ \ \ \ \ \ \ \ \ \ \ \ \times (dz^*_t-C\hat{x}^*_t\,dt -\psi ^2_{t,0}\,dt),\\
			\psi ^1_{0,\theta  }=\psi ^1_{t,-\varepsilon }=0,\ -\varepsilon \le \theta \le 0,\ t\ge 0,
		\end{cases}
	\end{equation}
	\begin{equation}
		\label{4.5}
		\begin{cases}
			\big( \frac{\partial }{\partial t}+\frac{\partial }{\partial \alpha }\big) \psi ^2_{t,\alpha }\,dt=\big( M_{t,\alpha }C^\top +\Lambda ^{22} _{t,-\alpha }-N_{t,\alpha ,0}\big) \\
			\ \ \ \ \ \ \ \ \ \ \ \ \ \ \ \ \ \ \ \ \ \ \ \ \ \times (dz^*_t-C\hat{x}^*_t\,dt -\psi ^2_{t,0}\,dt),\\
			\psi ^2_{0,\alpha }=\psi ^2_{t,-\varepsilon }=0,\ -\varepsilon \le \alpha \le 0,\ t\ge 0.
		\end{cases}
	\end{equation}
	Additionally, $K$ is a unique solution of the backward Riccati equation in \eqref{2.7}, $P$ is a unique solution of the forward Riccati equation
	\begin{equation}
		\label{4.6}
		\begin{cases}P'_t =AP_t+P_tA^\top +Q_{t,0}+Q^\top _{t,0}-(P_tC^\top +M^\top _{t,0})(CP_t+M_{t,0}),\\ P_0=\mathrm{cov}\,\xi ,\ t>0,  \end{cases}
	\end{equation}
	$Q$, $M$, $R$, $N$, and $S$ are mild solutions of 
	\begin{equation}
		\label{4.7}
		\begin{cases}
			\big( \frac{\partial }{\partial t}+\frac{\partial }{\partial \theta} \big) Q_{t,\theta }=Q_{t,\theta }A^\top +\Lambda ^{11}_{t,-\theta }-R_{t,\theta ,0}\\
			\ \ \ \ \ \ \ \ \ \ \ \ \ \ \ \ \ \ \ \ \ -\big( Q_{t,\theta }C^\top +\Lambda ^{12}_{t,-\theta }-S_{t,\theta ,0}\big) (C\!P_t+M_{t,0}),\! \\ 
			Q_{0,\theta }=Q_{t,-\varepsilon }=0,\ -\varepsilon \le \theta \le 0,\ t\ge 0, 
		\end{cases}
	\end{equation}
	\begin{equation}
		\label{4.8}
		\begin{cases}
			\big( \frac{\partial }{\partial t}+\frac{\partial }{\partial \alpha }\big) M_{t,\alpha }=M_{t,\alpha }
			A^\top +\Lambda ^{12} _{t,-\alpha }-S_{t,\alpha ,0}\\
			\ \ \ \ \ \ \ \ \ \ \ \ \ \ \ \ \ \ \ \ \ \ -\big( M_{t,\alpha }C^\top +\Lambda ^{22} _{t,-\alpha }-N_{t,\alpha ,0}\big) (CP_t+M_{t,0}),\\
			M_{0,\alpha }=M_{t,-\varepsilon }=0\ -\varepsilon \le \alpha \le 0,\ t\ge 0, 
		\end{cases}
	\end{equation}
	\begin{equation}
		\label{4.9}
		\begin{cases}
			\big( \frac{\partial }{\partial t}+\frac{\partial }{\partial \theta}+\frac{\partial }{\partial \tau }\big) R_{t,\theta ,\tau }=\big( Q_{t,\theta }C^\top +\Lambda ^{12}_{t,-\theta }-S_{t,\theta ,0}\big) \\
			\ \ \ \ \ \ \ \ \ \ \ \ \ \ \ \ \ \ \ \ \ \ \ \ \ \ \ \ \ \times \big( CQ^\top _{t,\tau }+\Lambda ^{12\top }_{t,-\tau }-S^\top _{t,\tau ,0}\big) ,\\
			R_{0,\theta ,\tau }= R_{t,-\varepsilon ,\tau }=R_{t,\theta ,-\varepsilon }=0,\ -\varepsilon \le \theta \le 0,\ -\varepsilon \le \tau \le 0,\ t\ge 0,\!\!\! 
		\end{cases}
	\end{equation}
	\begin{equation}
		\label{4.10}
		\begin{cases}
			\big( \frac{\partial }{\partial t}+\frac{\partial }{\partial \alpha}+\frac{\partial }{\partial \sigma }\big) N_{t,\alpha ,\sigma }=\big( M_{t,\alpha }C^\top +\Lambda ^{22} _{t,-\alpha }-N_{t,\alpha ,0}\big) \\
			\ \ \ \ \ \ \ \ \ \ \ \ \ \ \ \ \ \ \ \ \ \ \ \ \ \ \ \ \ \ \times \big( CM^\top _{t,\sigma }+\Lambda ^{22\top } _{t,-\sigma }-N^\top _{t,\sigma ,0}\big) ,\\
			N_{0,\alpha ,\sigma }=N_{t,-\varepsilon ,\sigma }\!=\!N_{t,\alpha ,-\varepsilon }\!=\!0,\ -\varepsilon \le \alpha \le 0,\ -\varepsilon \le \sigma \le 0,\ t\ge 0, \!\!\!
		\end{cases}
	\end{equation}
	\begin{equation}
		\label{4.11}
		\begin{cases}
			\big( \frac{\partial }{\partial t}+\frac{\partial }{\partial \theta}+\frac{\partial }{\partial \alpha }\big) S_{t,\theta ,\alpha }=\big( Q_{t,\theta }C^\top +\Lambda ^{12}_{t,-\theta }-S_{t,\theta ,0}\big) \\
			\ \ \ \ \ \ \ \ \ \ \ \ \ \ \ \ \ \ \ \ \ \ \ \ \ \ \ \ \ \times \big( CM^\top _{t,\alpha }+\Lambda ^{22\top } _{t,-\alpha}-N^\top _{t,\alpha ,0}\big) ,\\
			S_{0,\theta ,\alpha }=S_{t,-\varepsilon ,\alpha }=S_{t,\theta ,-\varepsilon }=0,\ -\varepsilon \le \theta \le 0,\ -\varepsilon \le \alpha \le 0,\ t\ge 0.\!\!\!\!
		\end{cases}
	\end{equation}
	
	\subsection{Sketch of the proof}
	
	The idea of the derivation is similar to the case of Example 1. We let
	\begin{equation}
		\label{4.12}
		\varphi ^1_t=\int _{\max (0,t-\varepsilon )}^t\Phi ^1_{t,s-t}\,dw_s\ \text{and}\  \varphi ^2_t=\int _{\max (0,t-\varepsilon )}^t\Phi ^2_{t,s-t}\,dw_s 
	\end{equation}
	where $w$ is an $l$-dimensional standard Wiener process. $\Phi ^1$ and $\Phi ^2$ are found from 
	\[
	\mathrm{cov}\bigg( \begin{bmatrix} \varphi ^1_{t+\theta }\\ \varphi^2_{t+\theta }\end{bmatrix},\begin{bmatrix} \varphi ^1_t\\ \varphi^2_t\end{bmatrix} \bigg)=
	\begin{bmatrix} \Lambda ^{11}_{t,\theta } & \Lambda ^{12}_{t,\theta } \\ \Lambda ^{12*}_{t,\theta } & \Lambda ^{22}_{t,\theta }\end{bmatrix} 
	\]
	which has infinitely many solutions. Chose a solution $(\Phi ^1, \Phi ^2)$ so that  
	\[
	\Phi ^1\in L_2\big( 0,T;L_2\big( -\varepsilon ,0;\mathbb{R}^{n\times l}\big) \big)
	\]
	while 
	\[
	\frac{\partial \Phi ^2}{\partial t},\frac{\partial \Phi ^2}{\partial \theta }\in L_2\big( 0,T;L_2\big( -\varepsilon ,0;\mathbb{R}^{k\times l}\big) \big)\ \text{and}\ \Phi ^2_{t,-\varepsilon }=0.
	\] 
	This is main distinction applied to BNs corrupting state and observations. While for a BN of the state system, it suffices to run the proof with an $L_2$-type relaxing function $\Phi ^1$, it turns to be suitable to handle a smooth relaxing function $\Phi ^2$ with $L_2$-type partial derivatives for a BN of the observation system. Then letting 
	\[
	\tilde{\varphi }^1_{t,\theta }=\int _{\max (0,t-\theta -\varepsilon )}^t\Phi ^1_{t-\theta ,s-t+\theta }\,dw_s
	\]
	and
	\[ 
	\tilde{\varphi }^2_{t,\alpha }=\int _{\max (0,t-\alpha -\varepsilon )}^t\frac{\partial }{\partial \alpha }\Phi ^2_{t-\alpha ,s-t+\alpha }\,dw_s,
	\] 
	we obtain equations of the type \eqref{3.9} for $\tilde{\varphi }^1$ and $\tilde{\varphi }^2$ with $\Gamma \tilde{\varphi }^1_{t,\cdot }=\tilde{\varphi }^1_{t,0}=\varphi ^1_t$ and
	\[
	\Pi \tilde{\varphi }^2_{t,\cdot }=\int _{-\varepsilon }^0\tilde{\varphi }^2_{t,\theta }\,d\theta =\varphi ^2_t.
	\]
	While $\Gamma $ is an unbounded linear operator, $\Pi $ is bounded. The rest of the proof is replacement the state process $x_t$ with a new state process
	\[
	\tilde{x}_t=\begin{bmatrix} x_t \\ \tilde{\varphi }^1_{t,\cdot } \\ \tilde{\varphi }^2_{t,\cdot } \end{bmatrix},
	\]
	rewriting the cost functional in \eqref{2.1} and observation process in terms of new state process, applying the separation principle and Kalman filtering in infinite dimension, and then getting equations for the components of the best estimate and related Riccati equations. Note that in the new setting, the observation process becomes
	\[
	dz_t=\begin{bmatrix} C & 0 & \Pi \end{bmatrix}\begin{bmatrix} x_t \\ \tilde{\varphi }^1_{t,\cdot } \\ \tilde{\varphi }^2_{t,\cdot } \end{bmatrix}dt+dv_t
	\]  
	which fits to infinite dimensional Kalman filter because $\Pi $ is a bounded linear operator.
	
	\subsection{Discussion}
	
	\eqref{2.15}, \eqref{4.2}--\eqref{4.11} and \eqref{2.7} form a complete set of equations for the optimal control in the LQG problem for the cost functional from \eqref{2.1} and the BN driven state and observation systems from \eqref{4.1}. These equations have the following meaning.
	\begin{itemize}
		\item 
		At first, note that the innovation process $\bar{z}$ changes its equation in comparison to the case of Example 1. Now, it is defined through the BN $\psi ^2_{t,0}$ in the form
		\[
		d\bar{z}_t=dz^*_t-C\hat{x}^*_t\,dt -\psi ^2_{t,0}\,dt,\ \bar{z}_0=0,\ t>0.
		\] 
		According to infinite dimensional Kalman filtering, it is a Wiener process.
		\item
		Equations in \eqref{4.4} and \eqref{4.5} compared to \eqref{3.9} tell us that $\psi ^1_{t,0}$ and $\psi ^2_{t,0}$ are BNs generated by the Wiener process $\bar{z}$. So, the equation in \eqref{4.2} for the best estimate $\hat{x}^*_t$ is driven by two BNs: Directly, by $\psi ^1_{t,0}$ and indirectly through the innovation process $\bar{z}$ by $\psi ^2_{t,0}$ which are created by \eqref{4.4} and \eqref{4.5}.
		\item
		The equation in \eqref{2.7} is a backward Riccati equation associated with control problem while the equation in \eqref{4.6} is a modification of the forward Riccati equation from Kalman filtering. 
		\item
		The equations in \eqref{4.9} for $R$ and \eqref{4.10} for $N$ generate the autocovariance functions $\Sigma ^{11}_{t,\theta } =R_{t,-\theta ,0}$ of $\psi ^1 _{t,0}$ and $\Sigma ^{22}_{t,\alpha } =N_{t,-\alpha ,0}$ of $\psi ^2 _{t,0}$. At the same time, \eqref{4.11} creates the cross covariance function $\Sigma ^{12}_{t,\theta  }=S_{t,-\theta ,0}$ of $\psi ^1_{t,0}$ and $\psi ^2_{t,0}$.
		\item 
		The relaxing functions of BN's $\psi ^1_{t,0}$ and $\psi ^2_{t,0}$ have the exact forms, respectively,  
		\[
		\Psi ^1_{t,\theta }=Q_{t+\theta ,\theta }C^\top +\Lambda ^{12}_{t-\theta ,-\theta }-S_{t,\theta ,0}
		\]
		and 
		\[
		\Psi ^2_{t,\alpha }=M_{t+\alpha ,\alpha }C^\top +\Lambda ^{22}_{t-\alpha ,-\alpha }-N_{t,\alpha ,0}.
		\]
		$\Psi ^1$ is created by the solution $Q$ of \eqref{4.7} and difference of the cross covariance functions $\Lambda ^{12}$ (of $\varphi ^1_{t,0}$ and $\varphi ^2_{t,0}$) and $S$ (of $\psi ^1_{t,0}$ and $\psi ^2_{t,0}$). In a similar way, $\Psi ^2$ is created by the solution $M$ of \eqref{4.8} and difference of the autocovariance functions of the BNs $\varphi ^2_{t,0}$ and $\psi ^2_{t,0}$. 
		\item
		The innovative noise $\phi $ making the system \eqref{4.1} causal is defined by
		\[
		\phi '_t=\varphi _t-BG^{-1}B^\top \hat{\alpha }_t
		\]
		or
		\[
		\phi '_t=\varphi _t-BG^{-1}B^\top \int _t^{\min (t+\varepsilon ,T)}\mathcal{U}^\top _{s,t}K_s\psi ^2_{t,t-s}\,ds.
		\]
		\item 
		Again the equations in (2.15), (4.2)–(4.11) and (2.7) of the optimal control are independent on the variations of the relaxing functions $\Phi ^1$ of $\varphi ^1$ and $\Phi ^2$ of $\varphi ^2$. They depend only on the autocovariance and cross covariance functions of $\varphi ^1$ and $\varphi ^2$. However, for both BN $\psi ^1_{t,0}$ and $\psi ^2_{t,0}$ corrupting the equation of the optimal state these equations create the relaxing function $\Psi ^1_{t,\theta }$ and $\Psi ^2_{t,\alpha }$. \end{itemize}
	
	
	\section{Case 3: State is corrupted by pointwise delayed WN}
	\label{S.5}
	
	\subsection{Setting of problem}
	
	Now, consider LQG problem of minimizing \eqref{2.1} subject to the system 
	\begin{equation}
		\label{5.1}
		\begin{cases} 
			x_t'=Ax_t+Bu_t+Dw'_{\max (0,\lambda _t)},\ x_0=\xi ,\ 0<t\le T, \\
			z'_t=Cx_t+w'_t,\ z_0=0,\ 0<t\le T.
		\end{cases}
	\end{equation}
	In addition to the conditions of Theorem \ref{T.2.1}, we assume that $w$ is a $k$-dimensional standard Wiener process, $D\in \mathbb{R}^{n\times k}$, $w$ and $\xi $ are independent, $\xi $ is Gaussian, $\mathbf{E}\,\xi =0$, and $\lambda $ is a continuous and strictly increasing function with $t-\varepsilon \le \lambda _t\le t$. In \eqref{5.1}, the signal noise is a pointwise delay of  the observation noise in time for the time-dependent value $t-\lambda _t$. Instead of differential notation, the system in \eqref{5.1} is presented in the derivative form which is suitable to present the respective equations. Note that 
	\[
	w'_{\max (0,\lambda _t)}=(w')_{\max (0,\lambda _t)},
	\]
	not being confused with $(w_{\max (0,\lambda _t)})'$.
	
	\subsection{Optimal control}
	
	A combination of Theorem \ref{T.2.1} and the filtering result from \cite{B1, B9} produces the existence of a unique optimal control $u^*$ minimizing the cost functional in \eqref{2.1} subject to \eqref{5.1}. It is defined by \eqref{2.15} where $\hat{x}^*$ is a unique solution of the equation
	\begin{equation}
		\label{5.2}
		\begin{cases}
			d\hat{x}^*_t=(A\hat{x}^*_t+Bu^*_t+\psi _{t,0})\,dt+P_tC^\top (dz^*_t-C\hat{x}^*_t\,dt),\\
			\hat{x}^*_0=0,\ 0<t\le T,
		\end{cases}
	\end{equation}   
	and $\hat{\alpha }$ has the explicit form
	\begin{equation}
		\label{5.3}
		\hat{\alpha }_t=\int _t^{\min (t+\varepsilon ,T)}\mathcal{U}^\top _{s,t}K_s\psi _{t,t-s}\,ds,
	\end{equation}
	provided that $\mathcal{U}_{t,s}$ is the transition matrix of $A-BG^{-1}B^\top K_t$. Note that \eqref{5.2} and \eqref{5.3} are same as \eqref{3.2} and \eqref{3.3}, respectively. However, the process $\psi $ included to these equations come from different sources. The process $\psi $ in \eqref{5.2} and \eqref{5.3} is a unique mild solution of
	\begin{equation}
		\label{5.4}
		\begin{cases} 
			\big(\frac{\partial }{\partial t}+\frac{\partial }{\partial \theta }\big) \psi _{t,\theta }=Q_{t,\theta }C^\top (z^{*\prime }_t-C\hat{x}^*_t),\\
			\psi _{t,\theta }=0,\ t-\lambda ^{-1}_0\le \theta \le 0,\ 0\le t\le \lambda ^{-1}_0,\\
			\psi _{t,t-\lambda ^{-1}_t}=D(z^{*\prime }_t-C\hat{x}^*_t),\ t>0,
		\end{cases}
	\end{equation}
	which differs from \eqref{3.4}. Similarly, $P$ is a unique solution of the Riccati equation
	\begin{equation}
		\label{5.5}
		\begin{cases}
			P'_t=AP_t+P_tA^\top +Q_{t,0}+Q^\top _{t,0}-P_tC^\top CP_t,\\
			P_0=\mathrm{cov}\,\xi ,\ t>0,
		\end{cases}
	\end{equation}
	which is same as \eqref{3.5} while the function $Q$ as well as $R$ are unique mild solutions of the equations
	\begin{equation}
		\label{5.6}
		\begin{cases} 
			\big( \frac{\partial }{\partial t}+\frac{\partial }{\partial \theta} \big)Q_{t,\theta }=Q_{t,\theta }A^\top -R_{t,\theta ,0}-Q_{t,\theta }C^\top CP_t,\\
			Q_{t,\theta }=0,\ t-\lambda ^{-1}_0\le \theta \le 0,\, 0\le t\le \lambda ^{-1}_0, \\
			Q_{t,t-\lambda ^{-1}_t}=-DCP_t,\ t>0,
		\end{cases}
	\end{equation}
	and
	\begin{equation}
		\label{5.7}
		\begin{cases} 
			\big( \frac{\partial }{\partial t}+\frac{\partial }{\partial \theta}+\frac{\partial }{\partial \tau }\big) R_{t,\theta ,\tau }=Q_{t,\theta }C^\top CQ^\top _{t,\tau },\\
			R_{t,\theta ,\tau }=0,\ t-\lambda ^{-1}_0\le \theta \le \tau \le 0,\ 0\le t\le \lambda ^{-1}_0,\\
			R_{t,\theta ,\tau }=0,\ t-\lambda ^{-1}_0\le \tau \le \theta \le 0,\ 0\le t\le \lambda ^{-1}_0,\\
			R_{t,\theta ,t-\lambda ^{-1}_t}=Q_{t,\theta }C^\top D^\top ,\ t-\lambda ^{-1}_t\!\le \!\theta \!<\!\min \big( 0,t-\lambda ^{-1}_0\big) ,\ t>0,\\
			R_{t,t-\lambda ^{-1}_t,\tau }=DCQ^\top _{t,\tau },\ t-\lambda ^{-1}_t\!\le \!\tau \!<\!\min \big( 0,t-\lambda ^{-1}_0\big) ,\ t>0.
		\end{cases}
	\end{equation}
	which differ from \eqref{3.6} and \eqref{3.7}, respectively.
	
	Equations in \eqref{5.4}, \eqref{5.6}, and \eqref{5.7} are partial differential equations. Their solutions are understood in the mild sense. For \eqref{5.4}, the mild solution has the explicit form  
	\[
	\psi _{t,\theta }=D(z^{*\prime }_{\lambda _{t-\theta}}-C\hat{x}^*_{\lambda _{t-\theta }})+\int _{\lambda _{t-\theta }}^tQ_{s,s-t+\theta }C^\top \,(z^{*\prime }_s-C\hat{x}^*_s)\,ds. 
	\]
	However, for \eqref{3.6} and \eqref{3.7}, the mild solutions are the solutions of the integral equations 
	\begin{align*}
		Q_{t,\theta }
		&=-DCP_{\lambda _{t-\theta }}
		&+\int _{\lambda _{t-\theta }}^t(Q_{s,s-t+\theta }A^\top -R_{s,s-t+\theta ,0}-Q_{s,s-t+\theta }C^\top CP_s) 
		\,ds 
	\end{align*}
	and
	\[
	R_{t,\theta ,\tau }=DCQ^\top _{\lambda _{t-\theta },\lambda _{t-\theta }-t+\tau }+\int _{\lambda _{t-\theta }}^tQ_{s,s-t+\theta }C^\top CQ^\top _{s,s-t+\tau }\,ds
	\]
	if $t-\lambda ^{-1}_t\le \theta \le \tau \le \min (0,t-\lambda ^{-1}_0)$
	and
	\[
	R_{t,\theta ,\tau }=Q_{\lambda _{t-\tau },\lambda _{t-\tau }-t+\theta }C^\top D^\top +\int _{\lambda _{t-\tau }}^tQ_{s,s-t+\tau }C^\top CQ^\top _{s,s-t+\theta }\,ds
	\]
	if $t-\lambda ^{-1}_t\le \tau \le \theta \le \min (0,t-\lambda ^{-1}_0)$.
	
	\subsection{Sketch of the proof}
	
	Let $\delta $ be the Dirac's delta-function. It turns out that if $\Phi _{t,\theta }=D\delta _\theta $ in \eqref{3.8}, then the LQG problem and its solution from Section \ref{S.3} reduce to the classic separation principle together with its ingredient Kalman filtering result for independent WNs. Therefore, the formulae of Section \ref{S.3} accept relaxing functions defined through the delta-function as well.  
	
	Based on this idea, the following guides for the proof of \eqref{2.15} and \eqref{5.2}--\eqref{5.7} can be determined. At first, following to the proof sketched in Subsection \ref{S.3.3}, get equations in the case of  correlated BN $\varphi $ and WN $v'$ by taking $w=v$. This results with equations in which the correlation terms appear. For simplicity, in Section \ref{S.3} the case of independent BN and WN was considered. Don't make passage from the relaxing function $\Phi $ to the autocovariance function $\Lambda $ but let $\Phi _{t,\theta }=D\delta _{\theta +t-\lambda _t}$ in the obtained equations instead of $\Phi _{t,\theta }=D\delta _\theta $. Then make simplification which moves correlation terms to the boundary conditions.
	
	\subsection{Discussion}
	
	The following are some comments regarding the LQG problem of this section.
	\begin{itemize}
		\item 
		The innovation process $\bar{z}$ defined by 
		\[
		\bar{z}'_t=z^{*\prime }_s-C\hat{x}^*_s,\ \bar{z}=0,
		\] 
		is still a standard Wiener process.
		\item The process $\psi _{t,0}$ is now equals to
		\[
		\psi _{t,0}=D\bar{z}'_{\lambda _t}+\int _{\lambda _t}^tQ_{s,s-t}C^\top \bar{z}'_s\,ds.
		\]
		Therefore, it is sum of a shifted WN and a BN.
		\item 
		Although the system in \eqref{5.1} is disturbed only by WNs, the shift of noises makes \eqref{5.2} to be disturbed by the WN $\bar{z}'_t$, its shift $\bar{z}'_{\lambda _t}$, and the BN component of $\psi _{t,0}$. 
		\item 
		A particular case of $\lambda $ is achieved if $\lambda _t=t-\varepsilon $. This is a case when the delay is a constant. Turning back to the scenario of spacecraft communication considered in the introductory section, a constant delay corresponds to a spacecraft which has approximately constant distance from the Earth. These are the Earth orbiting satellites and a spacecraft that either has landed on the Moon or is in the Lunar Orbit. The most far of them is the Moon which is 384.400 km away from the Earth. Radio signals with the approximate speed 300.000 km/s cover the distance  from a ground radar to the Moon  and come back for 2,56 s. This time delay is looking small, but it should be noted that it increases with time duration necessary for preparing return signals. Therefore, the case of constant delay could be useful for communication with a spacecraft that has either landed on the Moon or is in Lunar Orbit as well as with the Earth orbiting satellites.
		\item 
		The other planets of the Solar System are significantly distanced from the Earth. Spacecraft flying towards these planets with the interplanetary mission gradually increase the distance from the Earth. The spacecraft Voyager 2 lunched by NASA in 1977 passed to the Jovian, Saturnian, Uranian, Neptunian  systems and left the Solar System in 2018. The shortest distances to these systems in the increasing order are approximately 0.8, 1.5, 2.8, and 4.4 billion kilometers. Respectively, radio signals cover these distances and come back for approximately 1.5, 2.8, 5.2, and 8.1 hours, demonstrating a significant and gradually increasing time delay. This is the case when $\lambda _t=ct$. Here, $c$ satisfies $0<c <1$ to get $ct<t$. Also, $c$ must be close to 1 from the left because the speed of the spacecraft is much smaller than the speed of the radio signals. Therefore, the difference $t-ct=(1-c)t$ increases slowly.
		
		\item 
		One of the goals of the NASA's Mars Exploration Program (MEP)  is a human landing on the Mars. Respectively, a spacecraft exploring the near planet Mars should be designed with aim to return back. The shortest distance to Mars is approximately 54.600.000 km. Radio waves run this distance and come back for approximately 6 min. This time delay is sufficiently large to be taken into account. For spacecraft flying towards the Earth, the distance gradually decreases. This is the case $\lambda _t=c(t-\varepsilon)$. Here, $T=\varepsilon c(c-1)^{-1}$ is the instant when the spacecraft reaches the Earth because $c(T-\varepsilon )=T$. Also $c>1$ should be close to 1 because the difference $t-c(t-\varepsilon )=(1-c)t-c\varepsilon $ should decrease slowly.
	\end{itemize}
	
	
	\section{Conclusions}
	
	In this paper, we try to expose existing exact results for optimal control in the acausal LQG problems. Specifically, it was made an effort to demonstrate possible areas of application of these results by discussing different scenarios. A further development of this issue is seen as follows:
	\begin{itemize}
		\item
		Theoretically, optimal control problems for acausal systems can be expanded to nonlinear systems. In this regard, the separation of controls into two terms from Subsection \ref{S.2.3} so that one of them makes the system causal and the other one minimizes the cost functional in the causal setting can play a crucial role. For LQG problems, this issue has been solved perfectly. For nonlinear acausal systems this problem stays open.
		\item
		From applied point of view, the presented results are ready for application. Just they need a support by developing some numerical methods for solution of nonlinear components. The paper contains some scenarios of possible areas of application.   
	\end{itemize} 
	
	
	\section*{Acknowledgements} The first author acknowledges support by the Open
	Access Publication Fund of the University of Duisburg-Essen. This work has been funded by the Deutsche Forschungsgemeinschaft (DFG, GermanResearch Foundation) through the research grant number HU1889/7-1. 
	


\begin{thebibliography}{99}
		
		\bibitem{YZ}
		J. Yong and X.Y. Zhou,
		{\it Stochastic Controls: Hamiltonian Systems and HJB Equations}, 
		Springer, New York, 1999.
		
		\bibitem{FS}
		W. H. Fleming and H. M. Soner,
		{\it Controlled Markov Processes and Viscosity Solutions},
		Springer, New York, 2006.
		
		\bibitem{LZ}
		Q. L\"{u} and X. Zhang,
		{\it Mathematical Control Theory for Stochastic Partial Differential Equations},
		Springer, Cham, 2021. 
		
		\bibitem{B1}
		A. E. Bashirov,
		{\it Kalman-type filter for communication with considerably distanced spacecraft},
		JBIS, 47 (2021), pp. 381-385
		
		\bibitem{FR}
		W. H. Fleming and R. W. Rishel, 
		{\it Deterministic and Stochastic Optimal Control},
		Springer, New York, 1975.
		
		\bibitem{KR}
		H. J. Kushner and  W. J. Runggaldier, 
		{\it Filtering and control for wide bandwidth noise driven systems},
		IEEE Trans. Automat. Control, 32-AC (1987), pp. 123–133.
		
		\bibitem{CJ}
		J. L. Crassidis, J. L. Junkins,
		{\it Optimal Estimation of Dynamic Systems},
		CRC Press, Boca Raton, 2012.
		
		\bibitem{AA}
		L. J. S. Allen,
		{\it A primer on stochastic epidemic models: Formulation, numerical simulation, and analysis},
		Infectious Disease Modeling, 2(2) (2017), pp. 128--142.
		
		\bibitem{BP}
		G. B. Blanenship and G. C. Papanicolaou,
		{\it Stability and control of stochastic systems with wide band noise disturbances},
		SIAM J. Appl. Math., 34 (1978), pp. 437--476.
		
		\bibitem{K1}
		H. J. Kushner and W. J. Runggaldier,
		{\it Nearly optimal state feedback controls for stochastic systems with wide band noise},
		SIAM J. Control Optim., 25 (1987), pp. 298--315.
		
		\bibitem{K3} 
		H. J. Kushner and K. M. Ramachandran,
		{\it Nearly optimal singular controls for sideband noise driven systems},
		SIAM J. Control, 26 (1988), pp. 569--591. 
		
		\bibitem{LRT}
		R. S. Liptser, W. J. Runggaldier, and M. Taskar,
		{\it Diffusion approximation and optimal stochastic control},
		Theory Probab. Appl., 44 (2000), pp. 669--698.
		
		\bibitem{B2}
		A. E. Bashirov, L. V. Eppelbaum, and L.R. Mishne, 
		{\it Improving E\"{o}tv\"{o}s corrections by wide band noise Kalman filtering},
		Geophys. J. Int., 108 (1992), pp. 193--197.  
		
		\bibitem{B3}
		A. E. Bashirov, H. Etikan, and N. \c{S}emi,
		{\it Filtering, smoothing and prediction for wide band noise driven systems},
		J. Franklin Inst., 334B(4) (1997), pp. 667-683.
		
		\bibitem{B4}
		A. E. Bashirov and S. U\v{g}ural,
		{\it Representation of systems disturbed by wide band noise},
		{\it Appl. Math. Lett.}, 15 (2002), pp. 607--613.
		
		\bibitem{B5}
		A. E. Bashirov and S. U\v{g}ural,
		{\it Analyzing wide band noise processes with application to control and filtering},
		IEEE Trans. Automat. Control, 47(2) (2002), pp. 323--327.
		
		\bibitem{B6}
		A. E. Bashirov, Z. Mazhar, H. Etikan, and S. Ert\"{u}rk,
		{\it Delay structure of wide band noises with application to filtering problems}, 
		Optimal Control Appl. Methods, 34 (2013), pp. 69--79.
		
		\bibitem{SZY}
		E. Song, Y. Zhu, and Z. You, 
		{\it The Kalman type recursive state estimator with a finite-step correlated process noises}, 
		in: Proc. ICAL, Qingdao, China, 2008, pp. 196–200.
		
		\bibitem{JZZ}
		P. Jiang, J. Zhou, and Y. Zhu, 
		{\it Globally optimal Kalman filtering with finite-time correlated noises}, 
		in: Proc. 49th IEEE CDC, Atlanta, USA, 2010, pp. 5007–5012.
		
		\bibitem{LZW}
		L. Fan, J. Zhou, and D. Wu, 
		{\it Optimal filtering for systems with finite-step autocorrelated noises and multiple packet dropouts}, 
		Aerospace Science and Technology, 24(1) (2013), pp. 255–263.
		
		\bibitem{STL}
		S. Sun, T. Tian, and H. Lin, 
		{\it Optimal linear estimators for systems with finite-step correlated noises and packet dropout compensations}, 
		IEEE Trans. Signal Process., 64(21) (2016), pp. 5672–5681.
		
		\bibitem{HWG}
		J. Hu, Z. Wang, and H. Gao, 
		{\it Recursive filtering with random parameter matrices, multiple fading measurements and correlated noises}, 
		Automatica,  49(11) (2013), pp. 3440–3448.
		
		\bibitem{FWZ}
		J. Feng, Z. Wang, and M. Zeng, 
		{\it Optimal robust non-fragile Kalman-type recursive filtering with finite-step autocorrelated noises and multiple packet dropouts}, 
		Aerospace Science and Technology, 15(6) (2011), pp. 486–494.
		
		\bibitem{ZZWZ}
		S. Zhang, Y. Zhao, F. Wu, and J. Zhao, 
		{\it Robust recursive filtering for uncertain systems with finite-step correlated noises, stochastic nonlinearities and autocorrelated missing measurements}, 
		Aerospace Science and Technology, 39 (2014) pp. 272–280.
		
		\bibitem{LSZ}
		W. Liu, P. Shi, and H. Zhang, 
		{\it Kalman filtering with finite-step autocorrelated measurement noise}, 
		J. Comput. Appl. Math., 408 (2022), 114138.
		
		\bibitem{B7}
		A. E. Bashirov, 
		{\it Optimal control of partially observed systems with arbitrary dependent noises: linear quadratic case}, 
		Stochastics, 17 (1986), pp. 163--205.
		
		\bibitem{B12}
		A, E. Bashirov,
		{\it Filtering for linear systems with shifted noises},
		Int. J. Control, 78(7) (2005), pp. 521--529.
		
		\bibitem{B8}
		A. E. Bashirov,
		{\it Linear Filtering for Wide Band Noise Driven Observation Systems},
		Circuits Syst. Signal Process., 36 (2017), pp. 1247–1263 (doi: 10.1007/s00034-016-0355-y).
		
		\bibitem{B11}
		A. E. Bashirov, K. Abuasba,
		{\it Invariant filtering results for wide band noise driven signal systems}, 
		TWMS J. App. Eng. Math., 8(1) (2018), pp. 71--82. 
		
		\bibitem{B10}
		A. E. Bashirov, K. Abuassba,
		{\it Invariant Kalman filter for correlated wide band noises},
		Asian J. Control, 22(2) (2020), pp. 648–656 (doi.: 10.1002/asjc.1949).
		
		\bibitem{B9}
		K. Abuasbeh, A. E. Bashirov, 
		{\it Derivation of a Kalman-type filter for linear systems with pointwise delay in signal noise},
		Bound. Value Probl., 64 (2022), 18 p. (doi.: 10.1186/s13661-022-01646-6).
		
		\bibitem{BV}
		A. Bensoussan, M. Viot,
		{\it Optimal control of stochastic linear distributed parameter systems},
		SIAM J. Control, 13 (1975), pp. 904-926.
		
		\bibitem{CP}
		R. F. Curtain, A. J. Pritchard, 
		{\it Infinite Dimensional Linear Systems Theory, Lecture Notes in Control and Information Sciences, 8}, Springer, Berlin, 1978. 
		
	\end{thebibliography}
\end{document}